\documentclass[12pt,english,reqno]{amsart}
\usepackage[T1]{fontenc}
\usepackage{lmodern}
\usepackage{geometry}
\usepackage{amsmath,amssymb,amsthm,mathtools,comment}
\usepackage{amsfonts}
\usepackage[shortlabels]{enumitem}
\usepackage[pdftex,colorlinks,backref=page,citecolor=blue]{hyperref}
\usepackage[mathscr]{euscript}
\usepackage[usenames,dvipsnames]{color}
\usepackage{tikz,babel,adjustbox}
\usepackage[numbers]{natbib}
\usepackage{graphicx}
\usepackage{caption}
\usepackage{subcaption}
\usepackage{verbatim}
\usepackage{array}
\usepackage[frame,cmtip,arrow,matrix,line,graph,curve]{xy}
\usepackage{etoolbox}
\usepackage{pifont}
\usepackage[final]{microtype}
\usepackage{cmtiup}

\hypersetup{pdfpagemode=UseNone,pdfstartview={XYZ null null 1.00}}
\usetikzlibrary{shapes.misc,calc,intersections,patterns,decorations.pathreplacing}
\usetikzlibrary{arrows,arrows.meta,shapes,positioning,decorations.markings}
\tikzstyle arrowstyle=[scale=1]

\allowdisplaybreaks

\makeatletter
\def\@settitle{\begin{center}%
		\bfseries\Large \@title
	\end{center}%
}
\patchcmd{\@setauthors}{\MakeUppercase}{\normalsize}{}{}
\makeatother

\theoremstyle{plain}
\newtheorem{theorem}{Theorem}[section]
\newtheorem{lemma}[theorem]{Lemma}

\newtheorem{proposition}[theorem]{Proposition}
\newtheorem{corollary}[theorem]{Corollary}

\newtheorem{question}[theorem]{Question}
\theoremstyle{remark}

\newcommand{\beq}[1]{\begin{equation}\label{#1}}
\newcommand{\enq}[0]{\end{equation}}
\def\Prob{\mathbb{P}}

\makeatletter
\def\imod#1{\allowbreak\mkern10mu({\operator@font mod}\,\,#1)}
\makeatother
\newcommand{\eps}{\ensuremath{\varepsilon}}

\newcommand{\T}{\mathbb T}
\newcommand{\Z}{\mathbb Z}
\newcommand{\R}{\mathbb R}
\newcommand{\one}{\mathbf 1}
\newcommand{\Lip}{\operatorname{Lip}}
\newcommand{\SL}{\operatorname{SL}}

\newcommand{\nicefont}{\mathsf}
\newcommand{\mM}{\nicefont M}
\newcommand{\mN}{\nicefont N}
\newcommand{\supp}{\operatorname{supp}}
\newcommand{\dist}{\operatorname{dist}}
\newcommand{\diam}{\operatorname{diam}}

\newcommand{\op}{\mathrm{op}}

\begin{document}

\title{Thresholds for geometric graphs}
\author{Bhargav Narayanan}
\address{Department of Mathematics, Rutgers University, Piscataway, NJ 08854, USA}
\email{narayanan@math.rutgers.edu}

\date{20 May 2026}
\subjclass[2020]{Primary 05C80; Secondary 52A40, 60D05, 37A30}

\begin{abstract}
A metric probability space $\mM$ admits thresholds if the random geometric graph on $\mM$ has a threshold for every monotone graph property. We connect the existence of thresholds to the uniform expansion of $\mM$ and prove that all standard tori, spheres, and cubes admit thresholds.
\end{abstract}

\maketitle

\section{Introduction}\label{sec:introduction}

The random geometric graph is parametrised by its radius. The radius plays much the same role as the edge density in the Erd\H{o}s--R\'enyi random graph: as the radius grows, edges are added. This raises a basic question: does every monotone graph property appear in the random geometric graph at a well-defined scale? We answer affirmatively for the standard ambient spaces; the point is that geometric expansion can substitute for probabilistic independence.

Let $\mM=(M,d,\mu)$ be a metric probability space.  The random geometric graph $G_\mM(n,r)$ on $[n] =\{1,\dots,n\}$ is obtained by sampling points $x_1,\dots,x_n$ of $M$ independently from $\mu$ and then joining $i$ to $j$ when $d(x_i,x_j)\le r$.  A monotone graph property $\mathcal F$ on $[n]$ is any family of graphs on $[n]$ closed under adding edges; it need not be symmetric under relabelling the vertices.

A sequence of radii $r_n$ is a \emph{threshold} for a sequence of nontrivial properties $\mathcal F_n$ if
\[
    \Prob(G_\mM(n,s_n)\in\mathcal F_n)\to
    \begin{cases}
        0 & \text{if }s_n/r_n\to0,\text{ and}\\
        1 & \text{if }s_n/r_n\to\infty;
    \end{cases}
\]
here nontrivial means that $\mathcal F_n$ is neither empty nor the family of all graphs.  We say $\mM$ \emph{admits thresholds} if every sequence of nontrivial monotone properties has a threshold.

Before discussing geometry, it is helpful to recall what happens in the Erd\H{o}s--R\'enyi random graph $G(n,p)$. If a monotone property has positive probability at edge density $p$, then we may take several independent copies of $G(n,p)$ and unite their edge sets.  The union of $t$ such copies is another Erd\H{o}s--R\'enyi graph, now with edge density at most $tp$, and the probability of the monotone property is amplified because it is enough for one of the $t$ copies to have the property. This simple amplification argument is one way to see why every monotone graph property has a threshold in the sense of Erd\H{o}s--R\'enyi~\citep{ErdosRenyi1960}. The original proof of this fact, via an isoperimetric inequality for the hypercube, the Kruskal--Katona theorem~\citep{Kruskal1963,Katona1968}, is due to Bollob\'as and Thomason~\citep{BollobasThomason1987}.

For geometric graphs, there is no useful operation that simply takes a union of independent geometric graphs and stays in the same model. The substitute is to move the point configuration around.  Suppose $A\subset M^n$ is the set of configurations whose geometric graph has the desired property at radius $r$. If a measure-preserving map $\phi:M\to M$ sends a configuration of points $(x_1,\dots,x_n)$ into $A$, and if $\phi^{-1}$ stretches distances by at most $K$, then the original configuration has the property at radius $Kr$.  Thus a positive-probability event at radius $r$ becomes a high-probability event at radius $Kr$, \emph{if} a small family of measure-preserving maps can move most configurations into any positive-measure set. This motivates the expansion condition used in the paper.

Our main tool, with definitions to follow, is a geometric analogue of the Bollob\'as--Thomason theorem: independence can be replaced by a spectral gap for bounded-distortion, measure-preserving motions of the ambient space.

\begin{theorem}\label{thm:informal-expansion}
    If $\mM$ is regular and expands uniformly, then $\mM$ admits thresholds.
\end{theorem}

The technical work is therefore to prove uniform expansion for the spaces we care about. Our first application, for tori, is in effect an isoperimetric inequality.

\begin{theorem}\label{thm:main-tori}
    For all $d\ge1$, the torus $\T^d$, with its usual metric and Haar measure, admits thresholds.
\end{theorem}

Here the required uniform expansion comes from linear automorphisms.  Everything interesting already happens in dimension two where a large subgroup of $\SL_2(\Z)$ gives a spectral gap on the nonzero Fourier modes; using groups to find expansion is classical, with Ramanujan graphs~\citep{Margulis1988,LubotzkyPhillipsSarnak1988} a well-known example. In higher dimensions, several mechanisms are available; to keep the presentation self-contained, we place this two-dimensional construction inside each coordinate plane and average over the planes. The one-dimensional case is a theorem of McColm~\citep{McColm2004}, and is \emph{not proved} by this method: the circle has too few linear automorphisms.

Our third result is for spheres, and this requires one further idea.  Rotations preserve all pairwise distances, so they cannot by themselves mix configuration spaces in the way we need. The argument in dimension two is standard: we use the reduction via the pillowcase quotient $\T^2/\{x\sim -x\}$ to the torus $\T^2$ and then transfer uniform expansion back by a change of variables.  For higher-dimensional spheres, we use induction: expand along latitude spheres, rotate, and repeat.

\begin{theorem}\label{thm:main-spheres}
    For all $d\ge1$, the round sphere $S^d$, with its geodesic metric and uniform area measure, admits thresholds.
\end{theorem}

A few remarks on the one-dimensional case are also in order. McColm proved, as mentioned earlier, that the circle and the closed interval admit thresholds. This shows that uniform expansion is \emph{not necessary} for thresholds to exist; the proof uses a coupling with the underlying independence in first-passage percolation. On the other hand, the compact, one-dimensional space $[0,1]\cup[2,3]$, with the metric inherited from $\R$ and normalised Lebesgue measure, does \emph{not} admit thresholds: the monotone property of having a connected component with slightly more than half the vertices witnesses this failure.

\textbf{Organisation.}
Conventions are set out below, and some (safely skippable) background follows. Section~\ref{sec:expanding-maps} defines expansion and proves the basic amplification lemma. Section~\ref{sec:tori} proves uniform expansion for tori, Section~\ref{sec:cubes} gives a simple reduction from tori to cubes, and Section~\ref{sec:spheres} proves the spherical result. We conclude in Section~\ref{sec:conclusion} with a few questions about what comes next.

\textbf{Conventions.}
We write $L^2(\mathsf S)$ for the complex-valued square-integrable functions on a measure space $\mathsf S$, and $L^2_0(\mathsf S)$ for its mean-zero subspace; discrete spaces are also written this way, with counting measure understood.  The notation $\|\cdot\|_2$ is used for all such Hilbert-space norms; the ambient space will always make the meaning clear.  For a bounded operator $T$ on any Hilbert space considered here, we write $\|T\|_{\op}$ for the usual operator norm, i.e., the value $\sup_{\|x\|_2=1}\|Tx\|_2$.

\textbf{Background.}
Threshold phenomena are a central topic in the study of random discrete structures.  The cleanest general theory is for product measures on the hypercube, where independence gives access to isoperimetry, influences, and hypercontractivity, amongst other tools. These tools have led to a remarkably robust picture of monotone events in product spaces; see~\citep{janson} for background, and see~\citep{FriedgutKalai1996,Friedgut1999,FKNP,PP} for some highlights.

The obstacle in the present paper is precisely the one hidden by the product setting: the relevant random variables are not independent.  This difficulty is most familiar in statistical physics, where Gibbs measures replace independent bits by interacting configurations, and phase transitions (i.e., threshold behaviour) are often controlled by long-range dependence; see~\citep{FriedliVelenik2017} for example.  We know how to deal with this lack of independence in some notable instances; see~\citep{FKG,KoteckyPreiss1986,DuminilCopinRaoufiTassion2019}.  In many others, we still do not; the hard-sphere model~\citep{AlderWainwright1957} is a well-known example, but we claim no expertise.  The point of our result is that, for geometric graphs on the standard spaces, one can recover a form of amplification without recovering independence.

Random geometric graphs go back to Gilbert~\citep{Gilbert1961}, and many individual thresholds have been determined; see~\citep{Penrose2003} for an overview.  The lack of independence that complicates their analysis also makes geometric graphs a useful source of examples: the classical Bollob\'as--Erd\H{o}s~\citep{BollobasErdos1976} construction in Ramsey--Tur\'an theory is built from points on a high-dimensional sphere, and Ma--Shen--Xie~\citep{MaShenXie2025} recently used such graphs to obtain an exponential improvement in Ramsey lower bounds.

The question of whether the standard ambient spaces admit thresholds seems to have been folklore for some time. The question is natural enough that it likely predates the published record; spheres were recently highlighted by Perkins~\citep{Perkins2025}, and the earliest written reference we know of is McColm~\citep{McColm2004}.

\section{Expansion}\label{sec:expanding-maps}

This section contains our main definition, that of expansion, formalising the basic idea outlined earlier. With this definition in hand, it is easy to prove Theorem~\ref{thm:informal-expansion}.

Fix a metric probability space $\mM=(M,d,\mu)$.  For a configuration of points $\mathbf x=(x_1,\dots,x_n)\in M^n$ and a radius $r\ge0$, write $G_\mM(\mathbf x, r)$ for the graph on $[n]$ in which $i$ and $j$ are adjacent when $d(x_i,x_j)\le r$; the random graph $G_\mM(n,r)$ is obtained by sampling $\mathbf x \in M^n$ from $\mu^n$. If $\mathcal F$ is a monotone graph property on $[n]$, write
\[
\mathcal F_M(r)
=
\left\{\mathbf x\in M^n:G_\mM(\mathbf x, r)\in\mathcal F\right\}
\]
for the corresponding set of configurations.  Thus
\[
    \Prob(G_\mM(n,r)\in\mathcal F)=\mu^n(\mathcal F_M(r)).
\]

For $K\ge0$ and $0<\rho\le1$, we say that $\mM=(M,d,\mu)$ has \emph{$(K,\rho)$-expansion} if there is a finitely supported (to avoid measure-theoretic distractions) probability measure $\mathsf P$ on measure-preserving bijections $\phi:M\to M$ such that the following two conditions hold.
\begin{enumerate}[label=(E\arabic*)]
    \item \textbf{Bounded distortion.}  For every $\phi\in\supp \mathsf P$,
    \[
        \Lip(\phi)\le K
        \quad\text{and}\quad
        \Lip(\phi^{-1})\le K.
    \]

    \item \textbf{Spectral gap.}  For every $n\ge1$, the (family of) averaging operator(s)
    \[
        T_{\mathsf P} f(x_1,\dots,x_n)
        =
        \mathbb E_{\phi\sim\mathsf P}\left[f(\phi(x_1),\dots,\phi(x_n))\right]
    \]
    is contractive in the sense that
    \[
        \|T_{\mathsf P} f\|_2\le (1-\rho)\|f\|_2
    \]
    for every $f\in L^2_0(M^n)$.
\end{enumerate}

We say that \emph{$\mM$ expands uniformly} if it has $(K,\rho)$-expansion for some $\rho>0$ and some finite $K$; increasing $K$ only weakens the bounded-distortion condition, so (to avoid technicalities) we assume that $K>1$ in what follows. The parameter $K$ measures distortion, and $\rho$ is the spectral gap; what is important is that neither parameter depends on the number of points $n$.

It is fairly easy to see that expansion allows us to amplify probabilities in the following sense.

\begin{lemma}\label{lem:amplification}
    Suppose that $\mM=(M,d,\mu)$ has $(K,\rho)$-expansion.  Then for every pair of integers $n \ge 1$ and $t\ge 0$, every monotone graph property $\mathcal F$ on $[n]$, and every $r>0$ with $p=\mu^n(\mathcal F_M(r))>0$, we have
    \[
        \mu^n(\mathcal F_M(K^t r))
        \ge
        1-
        \frac{1-p}{p}(1-\rho)^{2t}.
    \]
\end{lemma}

\begin{proof}
    Let $F=\mathcal F_M(r)$, and let
    \[
        F_t=\left\{\mathbf x\in M^n:T_{\mathsf P}^t\one_F(\mathbf x)>0\right\}
    \]
    be the set of configurations that can be moved into $F$ by a composition of $t$ maps from $\supp \mathsf P$.

    First, we use bounded distortion.  If $\mathbf x\in F_t$, then the configuration $\Phi\mathbf x$ lies in $F$ for some composition $\Phi$ of $t$ maps from $\supp \mathsf P$.  Hence $G_\mM(\Phi\mathbf x, r)\in\mathcal F$.  Pulling the edges back through $\Phi$ scales distances by at most $K^t$, so
    \[
        G_\mM(\Phi\mathbf x, r)\subset G_\mM(\mathbf x, K^t r).
    \]
    Since $\mathcal F$ is monotone, this implies
    \begin{equation}\label{eq:Ft-contained}
        F_t\subset \mathcal F_M(K^t r).
    \end{equation}

    It remains to show that $F_t$ is large.  On $F_t^c$, we have $T_{\mathsf P}^t\one_F=0$ by definition, while the constant function $p$ is fixed by $T_{\mathsf P}$.  Hence the integrand $|T_{\mathsf P}^t\one_F-p|^2$ is $p^2$ on $F_t^c$, and integrating over $F_t^c$ yields
    \[
        \|T_{\mathsf P}^t\one_F-p\|_2^2
        \ge
        p^2\mu^n(F_t^c).
    \]
    Applying the spectral gap $t$ times to the mean-zero function $\one_F-p$ gives
    \[
        \|T_{\mathsf P}^t\one_F-p\|_2^2
        =
        \|T_{\mathsf P}^t(\one_F-p)\|_2^2
        \le
        (1-\rho)^{2t}\|\one_F-p\|_2^2
        =
        (1-\rho)^{2t}p(1-p),
    \]
    where the last equality comes from computing the variance of the indicator $\one_F$.  If $p > 0$, it follows that
    \[
        \mu^n(F_t)
        \ge
        1-\frac{1-p}{p}(1-\rho)^{2t};
    \]
    with~\eqref{eq:Ft-contained}, the lemma follows.
\end{proof}

Call $\mM$ \emph{regular} if, for every integer $n\ge1$ and every nontrivial monotone graph property $\mathcal F$ on $[n]$, the function $r\mapsto\mu^n(\mathcal F_M(r))$ is continuous, satisfies $\mu^n(\mathcal F_M(0))=0$, and satisfies $\mu^n(\mathcal F_M(r))=1$ for all sufficiently large $r$.

A standard argument using Lemma~\ref{lem:amplification} gives Theorem~\ref{thm:informal-expansion}.

\begin{proof}[Proof of Theorem~\ref{thm:informal-expansion}]
    Let $\mathcal F_n$ be a sequence of nontrivial monotone graph properties, and define
    \[
        r_n=\inf\left\{r:\mu^n((\mathcal F_n)_M(r))\ge1/2\right\}.
    \]
    Regularity gives $0<r_n<\infty$ and $\mu^n((\mathcal F_n)_M(r_n))=1/2$.

    Let $K$ and $\rho$ be expansion constants for $\mM$.  If $s_n/r_n\to\infty$, take
    \[
        t_n=\left\lfloor\log_K(s_n/r_n)\right\rfloor.
    \]
    Then $t_n\to\infty$ and $K^{t_n}r_n\le s_n$ for all large $n$, so Lemma~\ref{lem:amplification} and monotonicity give
    \[
        \mu^n((\mathcal F_n)_M(s_n))
        \ge
        \mu^n((\mathcal F_n)_M(K^{t_n}r_n))
        \ge
        1-(1-\rho)^{2t_n}\to1.
    \]

    Conversely, suppose that $s_n/r_n\to0$.  If $\mu^n((\mathcal F_n)_M(s_n))$ does not tend to $0$, then along a subsequence it is at least some fixed $\eps>0$.  On that subsequence set
    \[
        t_n=\left\lfloor\log_K(r_n/s_n)\right\rfloor-1.
    \]
    Then $t_n\to\infty$ and $K^{t_n}s_n<r_n$ for all large $n$.  Lemma~\ref{lem:amplification}, applied at radius $s_n$, gives
    \[
        \mu^n((\mathcal F_n)_M(K^{t_n}s_n))
        \ge
        1-\frac{1-\eps}{\eps}(1-\rho)^{2t_n}>1/2
    \]
    for all large $n$, contradicting the definition of $r_n$.  Thus $\mu^n((\mathcal F_n)_M(s_n))\to0$.
\end{proof}

Next, we record the fact that regularity is harmless for the spaces under consideration.

\begin{lemma}\label{lem:standard-regular}
    The tori $\T^d$, spheres $S^d$, and cubes $[0,1]^d$, equipped with their usual metrics and normalised measures, are regular.
\end{lemma}

\begin{proof}
    These spaces are compact and non-atomic.  If $\mathcal F$ is nontrivial and monotone, then the empty graph is not in $\mathcal F$ and the complete graph is in $\mathcal F$.  Hence $\mu^n(\mathcal F_M(0))=0$, and $\mu^n(\mathcal F_M(r))=1$ once $r\ge\diam M$.

    For continuity, fix $r$.  As the radius varies, the graph of a configuration can change only when the radius crosses one of the finitely many pairwise distances in that configuration.  The exceptional configurations with $d(x_i,x_j)=r$ for some pair $i<j$ have product measure zero in all three families of spaces.  Therefore the indicators of the sets $\mathcal F_M(r_k)$ converge almost everywhere to the indicator of $\mathcal F_M(r)$ whenever $r_k\to r$, and dominated convergence gives continuity.
\end{proof}

Finally, the following elementary transfer principle helps with the spherical argument.

\begin{lemma}\label{lem:bilip-expansion-transfer}
    Let $\mM=(M,d_M,\mu)$ and $\mN=(N,d_N,\nu)$ be metric probability spaces.  Suppose there is a measure-preserving bijection $\psi:M\to N$ and $0<a\le b<\infty$ such that
    \[
        a\,d_M(x,y)\le d_N(\psi x,\psi y)\le b\,d_M(x,y)
    \]
    for all $x,y\in M$.  If $\mM$ has $(K,\rho)$-expansion, then $\mN$ has $((b/a)K,\rho)$-expansion.
\end{lemma}

\begin{proof}
    For each map $\phi:M\to M$ witnessing the expansion of $\mM$, define
    \[
        \widetilde\phi=\psi\circ\phi\circ\psi^{-1}:N\to N;
    \]
    it is an easy exercise to check that these maps witness the expansion of $\mN$ with the required parameters.
\end{proof}

\section{Tori}\label{sec:tori}

We start with tori.  Translations by themselves do nothing to a geometric graph: if all points are shifted together, all pairwise distances stay exactly the same.  The maps we need must preserve measure, but they must also change distances in a controlled way.  On a torus, the useful maps are linear automorphisms.  They preserve measure, each one has bounded distortion, and averaging over a suitable finite family mixes the configuration space.

The proof is entirely in dimension two.  There, a free subgroup of the automorphisms supplies the spectral gap: on the Fourier side, the nonzero frequencies move along copies of the infinite regular tree which, famously, expands.  In higher dimensions, we apply this two-dimensional construction inside each coordinate plane and then average over the planes, so every nonzero frequency is detected by a positive fraction of the averages.  The circle is the exceptional case; it has too few linear automorphisms, and we rely instead on McColm's one-dimensional theorem.

Let $\T^d=\R^d/\Z^d$ denote the standard $d$-dimensional torus, with Haar measure and the usual flat metric.  Every matrix $A\in\SL_d(\Z)$ acts on $\T^d$ by
\[
    x\mapsto Ax \imod{\Z^d};
\]
this map preserves Haar measure since $\det(A) = 1$, and it is bi-Lipschitz with Lipschitz constants $\|A\|_{\op}$ and $\|A^{-1}\|_{\op}$.  We prove uniform expansion for all tori of dimension at least two by averaging over finitely many such automorphisms.

The argument is most transparent on the Fourier side.  On $(\T^d)^n$, the characters are indexed by $\mathbf v=(v_1,\dots,v_n)\in(\Z^d)^n$, and for $\mathbf x = (x_1,\dots,x_n)\in(\T^d)^n$, they are given by
\[
    \chi_{\mathbf v}(\mathbf x)
    =
    \exp\left(2\pi i\sum_{j=1}^n \langle v_j, x_j\rangle\right).
\]
The character with $\mathbf v=0$ is constant, and the remaining characters form an orthonormal basis of $L^2_0((\T^d)^n)$.  Given $A \in \SL_d(\Z)$, the key identity for its diagonal action on $(\T^d)^n$ is
\[
    \chi_{\mathbf v}(A \mathbf x)=\chi_{A^T\mathbf v}(\mathbf x),
\]
where $A \mathbf x = (Ax_1, \dots, Ax_n)$ and $A^T\mathbf v=(A^T v_1,\dots,A^T v_n)$.  Thus diagonal torus automorphisms act by permuting the nonzero Fourier frequencies.

For a finitely supported probability measure $\mathsf P$ on $\SL_d(\Z)$, the corresponding distortion constant
\[
    \max_{A\in\supp \mathsf P}
    \max\left\{\|A\|_{\op},\|A^{-1}\|_{\op}\right\}
\]
is always finite. Let $T_{\mathsf P}$ be the corresponding averaging operator on $L^2((\T^d)^n)$, namely
\[
    T_{\mathsf P} f(x_1,\dots,x_n)
    =
    \mathbb E_{A\sim\mathsf P}\left[f(Ax_1,\dots,Ax_n)\right].
\]

Writing $\Omega_n=(\Z^d)^n\setminus\{0\}$ for the space of nonzero frequencies, let $Q_n$ be the frequency-side operator on $L^2(\Omega_n)$ that, for $\mathbf{v} \in \Omega_n$, sends a square-integrable $h$ to $Q_n h$ defined by
\[
    Q_nh(\mathbf v)
    =
    \mathbb E_{A\sim\mathsf P}\left[h((A^T)^{-1}\mathbf v)\right].
\]
Indeed, if $f \in L^2_0((\T^d)^n)$ has the Fourier representation
\[
    f=\sum_{\mathbf v\in\Omega_n}\widehat f(\mathbf v)\chi_{\mathbf v},
\]
then the identity above gives
\[
    T_{\mathsf P} f=\sum_{\mathbf v\in\Omega_n}(Q_n\widehat f)(\mathbf v)\chi_{\mathbf v}.
\]

By Parseval, showing that $T_{\mathsf P}$ is contractive amounts therefore to controlling $\|Q_n\|_{\op}$ on the nonzero frequencies, uniformly in $n$.

We now choose the automorphisms.  For this, we use the following standard fact about Schottky subgroups of $\SL_2(\Z)$; see Beardon's book~\citep{Beardon1983} for a textbook presentation.
\begin{lemma}\label{lem:schottky}
    There are matrices $a,b\in\SL_2(\Z)$ that freely generate a subgroup $\Gamma\cong F_2$ such that no non-identity element of $\Gamma$ has eigenvalue $1$. \qed
\end{lemma}

What we need comes from the expansion of the Cayley graph of this free group, namely the infinite $4$-regular tree. The following computation is due to Kesten~\citep{Kesten1959}; see Kowalski's book~\citep{Kowalski2019} for a modern treatment.
\begin{theorem}\label{thm:kesten-tree}
    Let $H$ be the infinite $4$-regular tree. For the simple random walk operator $M$ on $L^2(H)$ that averages a function $h$ on the vertices of $H$ uniformly over neighbours, we have $\|M\|_{\op}={\sqrt{3}}/{2} < 1$. \qed
\end{theorem}

Let $a,b\in\SL_2(\Z)$ be the matrices supplied by Lemma~\ref{lem:schottky}, and put $S=\{a,a^{-1},b,b^{-1}\}$.  For $d\ge2$, let
\[
    \mathcal E_d=\left\{\{p,q\}:1\le p<q\le d\right\}
\]
be the set of coordinate two-planes.  If $e=\{p,q\}\in\mathcal E_d$ and $s\in S$, let $s{(e)}\in\SL_d(\Z)$ act by $s$ on the $p$ and $q$ coordinates, in that order, and act as the identity on the remaining coordinates.  Let $\mathsf P_d$ be the probability measure obtained by choosing $e\in\mathcal E_d$ uniformly, choosing $s\in S$ uniformly, and taking $s{(e)}$. Since $S=S^{-1}$, set
\[
    K_d=
    \max_{e\in\mathcal E_d}
    \max_{s\in S}
    \|s{(e)}\|_{\op}.
\]
This probability measure $\mathsf P_d$ witnesses the uniform expansion of $\T^d$, as we prove below.

\begin{proposition}\label{prop:torus-gap}
    For $d\ge2$, the torus $\T^d$ has $\left(K_d,1-\sqrt{1-1/(2d)}\right)$-expansion.
\end{proposition}

\begin{proof}
    By the Parseval discussion above, it suffices to prove a uniform spectral gap for the frequency-side operator on $L^2(\Omega_n)$. Let $\alpha=\sqrt3/2$ be the norm of the simple random walk operator on the four-regular tree.

    Fix a coordinate two-plane $e=\{p,q\}$.  Let $Q_{e,n}$ be the frequency-side average over the four matrices $s{(e)}$, $s\in S$, and let
    \[
        Z_e=\left\{(v_1,\dots,v_n)\in\Omega_n:
        (v_j)_p=(v_j)_q=0\text{ for every }j\right\};
    \]
    let $E_e$ be the orthogonal projection onto $L^2(Z_e)$.

    The frequencies in $Z_e$ have zero entries in the two coordinates of $e$, so all four moves leave them fixed.  Hence $Q_{e,n}$ is the identity on $L^2(Z_e)$.  We now look at the remaining frequencies as the vertex set of a graph $H$ in which we join each $\mathbf v\in\Omega_n\setminus Z_e$ to $((s{(e)})^T)^{-1}\mathbf v$ for each $s\in S$. Since $S=S^{-1}$, these are undirected edges, and $Q_{e,n}$ is just the simple random walk operator on this graph.

    We claim that each component of $H$ is the infinite four-regular tree.  Suppose, to the contrary, that a non-empty reduced word in the four steps brings some $\mathbf v\in\Omega_n\setminus Z_e$ back to itself.  The same word gives a non-identity element $\gamma\in\Gamma$ such that
    \[
        ((\gamma{(e)})^T)^{-1}\mathbf v=\mathbf v,
    \]
    where $\gamma{(e)}$ means that $\gamma$ is placed in the coordinates $p,q$ and the identity is placed in all other coordinates.  Since $\mathbf v\notin Z_e$, there is some $j$ for which the vector $u=((v_j)_p,(v_j)_q)$ is nonzero.  The displayed equality gives
    \[
        (\gamma^T)^{-1}u=u,
    \]
    or equivalently $\gamma^T u=u$.  It follows that $1$ is an eigenvalue of $\gamma$, contradicting Lemma~\ref{lem:schottky}; therefore there are no such closed reduced walks.  The four possible steps from each frequency are distinct, and the component of any starting frequency is therefore the Cayley graph of the free group on $a,b$, namely the infinite four-regular tree.

    When $d=2$, there is only one coordinate plane and $Z_e$ is empty, so Kesten's theorem already gives the desired contraction.  In higher dimensions, a single coordinate-plane average leaves the frequencies in $Z_e$ untouched, and the remaining step is to average over the coordinate planes and count how often a nonzero frequency is seen. We spell out the details below.

    Kesten's theorem applies on each component of $H$.  Therefore
    \[
        \|Q_{e,n}(I-E_e)h\|_2\le \alpha\|(I-E_e)h\|_2
    \]
    for each $h \in L^2(\Omega_n)$; since $Z_e$ and its complement are invariant under the coordinate-plane action, we also get
    \begin{equation}\label{eq:coordinate-plane-contraction}
        \|Q_{e,n}h\|_2^2
        \le
        \|h\|_2^2-(1-\alpha^2)\|(I-E_e)h\|_2^2.
    \end{equation}

    The full frequency-side average associated with $T_{\mathsf P_d}$ is
    \[
        Q_{n}=\frac1{|\mathcal E_d|}\sum_{e\in\mathcal E_d}Q_{e,n}.
    \]
    By convexity of the squared norm and~\eqref{eq:coordinate-plane-contraction},
    \[
        \|Q_{n}h\|_2^2
        \le
        \|h\|_2^2
        -(1-\alpha^2)
        \frac1{|\mathcal E_d|}\sum_{e\in\mathcal E_d}\|(I-E_e)h\|_2^2.
    \]

    It remains only to count how often a nonzero frequency is detected.  If $\mathbf v\in\Omega_n$, then some coordinate direction $r$ appears nontrivially in one of the vectors $v_j$.  Every two-plane containing $r$ detects $\mathbf v$, so
    \[
        \left|\left\{e\in\mathcal E_d:\mathbf v\notin Z_e\right\}\right|\ge d-1.
    \]
    Since $|\mathcal E_d|=\binom d2$, we have
    \[
        \frac1{|\mathcal E_d|}\sum_{e\in\mathcal E_d}\|(I-E_e)h\|_2^2
        =
        \sum_{\mathbf v\in\Omega_n}|h(\mathbf v)|^2
        \frac{\left|\left\{e:\mathbf v\notin Z_e\right\}\right|}{|\mathcal E_d|}
        \ge
        \frac2d\|h\|_2^2.
    \]
    As $1-\alpha^2=1/4$, this yields
    \[
        \|Q_{n}h\|_2^2
        \le
        \left(1-\frac1{2d}\right)\|h\|_2^2.
    \]
    The contraction factor is therefore, uniformly in $n$, at most $\sqrt{1-1/(2d)}$, proving the proposition.
\end{proof}

We now have everything we need to prove Theorem~\ref{thm:main-tori}.

\begin{proof}[Proof of Theorem~\ref{thm:main-tori}]
    The one-dimensional case is due to McColm, as mentioned earlier.  For every $d\ge2$, Proposition~\ref{prop:torus-gap} shows that $\T^d$ expands uniformly, Lemma~\ref{lem:standard-regular} supplies regularity, and Theorem~\ref{thm:informal-expansion} therefore gives thresholds.
\end{proof}

\section{Cubes}\label{sec:cubes}

The cube follows from the torus by comparison, as we show below.

\begin{corollary}\label{thm:main-cubes}
    For all $d\ge1$, the cube $[0,1]^d$, with its usual Euclidean metric and Lebesgue measure, admits thresholds.
\end{corollary}

We only use the elementary fact that the torus is the cube with opposite faces glued together.  Gluing can only make distances smaller; folding the torus back into the cube makes distances larger by at most a factor of $2$. We spell out the details below.

Write $C^d=[0,1]^d$ for the cube with Euclidean metric and normalised Lebesgue measure.  The following trivial sandwiching lemma gives Corollary~\ref{thm:main-cubes} from Theorem~\ref{thm:main-tori}.

\begin{lemma}\label{lem:cube-torus-sandwich}
    For every monotone graph property $\mathcal F$ on $[n]$ and every $r\ge0$,
    \begin{equation}\label{eq:cube-torus-sandwich}
        \mu_{C^d}^n(\mathcal F_{C^d}(r))
        \le
        \mu_{\T^d}^n(\mathcal F_{\T^d}(r))
        \le
        \mu_{C^d}^n(\mathcal F_{C^d}(2r)).
    \end{equation}
\end{lemma}

\begin{proof}
    For the first inequality, sample points in the fundamental cube $[0,1)^d$ and view them once in the cube and once in the torus.  Torus distance is never larger than cube distance, so every edge present in the cube at radius $r$ is also present in the torus at radius $r$.  Since $\mathcal F$ is monotone, the first inequality follows.

    For the other direction, fold the torus, as in Figure~\ref{fig:cube-folding}, into the cube by $\pi:\T^d\to C^d$ defined by
    \[
        \pi(x)_j=2\,\dist_{\T^1}(x_j,0).
    \]
    The push-forward of Haar measure under $\pi$ is Lebesgue measure on the cube.  Also,
    \[
        \|\pi(x)-\pi(y)\|_2\le 2d_{\T^d}(x,y)
    \]
    for all $x,y\in\T^d$.  Hence every torus edge of radius $r$ becomes a cube edge of radius $2r$ after folding.  Monotonicity now gives the second inequality.
\end{proof}

\begin{figure}[!t]
\centering
\begin{tikzpicture}[scale=1.0,>=Latex,every node/.style={font=\small}]
    \begin{scope}[xshift=0cm]
        \draw[thick] (0,0) circle (1.25);
        \coordinate (x) at (115:1.25);
        \coordinate (y) at (35:1.25);
        \draw[blue!70!black,very thick] (35:1.25) arc[start angle=35,end angle=115,radius=1.25];
        \filldraw[fill=blue!70!black,draw=white,line width=0.4pt] (x) circle (3pt) node[above left=1pt] {$x$};
        \filldraw[fill=blue!70!black,draw=white,line width=0.4pt] (y) circle (3pt) node[above right=1pt] {$y$};
        \filldraw[fill=white,draw=black,line width=0.5pt] (0,-1.25) circle (2.3pt);
        \node at (0,0.25) {$d_{\T^1}(x,y)\le r$};
        \node at (0,-1.85) {$\T^1$};
    \end{scope}

    \draw[->,very thick] (2.0,0) -- (3.35,0);
    \node at (2.68,0.35) {$\pi$};

    \begin{scope}[xshift=4.25cm]
        \draw[thick] (0,0) -- (3.25,0);
        \draw[thick] (0,-0.08) -- (0,0.08);
        \draw[thick] (3.25,-0.08) -- (3.25,0.08);
        \coordinate (px) at (0.75,0);
        \coordinate (py) at (2.35,0);
        \draw[blue!70!black,very thick] (0.75,0.20) -- (2.35,0.20);
        \filldraw[fill=blue!70!black,draw=white,line width=0.4pt] (px) circle (3pt) node[below=2pt] {$\pi(x)$};
        \filldraw[fill=blue!70!black,draw=white,line width=0.4pt] (py) circle (3pt) node[below=2pt] {$\pi(y)$};
        \node at (1.625,0.55) {$\le 2r$};
        \node at (1.625,-0.95) {$[0,1]$};
    \end{scope}
\end{tikzpicture}
\caption{The folding map in one coordinate.}
\label{fig:cube-folding}
\end{figure}
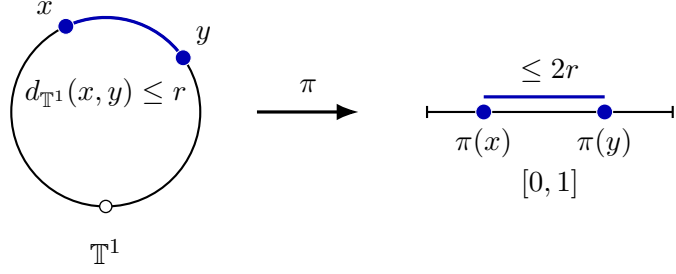

\section{Spheres}\label{sec:spheres}

We now turn to spheres.  Rotations by themselves do nothing to a geometric graph: if all points are rotated together, all pairwise distances stay exactly the same.  The maps we need must preserve measure, but they must also change distances in a controlled way.

The proof goes in two steps.  For the first step, we pass from $S^2$ to $\T^2$ through a low-dimensional coincidence. In the second step, we climb from $S^{m-1}$ to $S^m$: expand inside each latitude sphere, rotate the picture, and do it again. We are initially fortunate because we cannot start from $S^1$; while $S^1$ does not expand uniformly, it is also the already-addressed $\T^1$.

\textbf{The two-sphere.} Write
\[
    Q=\T^2/\{x\sim -x\}
\]
for the flat pillowcase.  As shown in Figure~\ref{fig:pillowcase}, the involution $x\mapsto -x$ is orientation-preserving and has four fixed points.  Away from those fixed points, $\T^2\to Q$ is a two-fold cover, and computing the Euler characteristic shows the underlying surface to be $S^2$.  The flat metric has cone angle $\pi$ at each of the four images of the fixed points.

With these facts in hand, we have what we need to pass from the torus to the pillowcase.

\begin{lemma}\label{lem:pillowcase-expands}
    The pillowcase $Q$, equipped with the quotient metric induced by the flat metric on $\T^2$ and the push-forward of Haar measure, expands uniformly.
\end{lemma}

\begin{proof}
    Use the four torus automorphisms from Proposition~\ref{prop:torus-gap}.  Each of them commutes with the map $x\mapsto -x$, so each one descends to a measure-preserving map of the quotient $Q$.  Passing to the quotient cannot increase the Lipschitz constant.

    The spectral gap also passes to the quotient.  Indeed, pulling a function on $Q^n$ back to $(\T^2)^n$ identifies $L^2(Q^n)$ with the subspace of $L^2((\T^2)^n)$ consisting of functions that are unchanged when any of the $n$ coordinates is replaced by its negative.  The torus averaging operator preserves this subspace.  Therefore the torus contraction on mean-zero functions restricts to the same contraction on $Q^n$.
\end{proof}

To pass from the pillowcase to the sphere, we need to replace the flat pillowcase metric by the round metric.  The only point is to do this without changing measure in an uncontrolled way. See the part on volume-preserving diffeomorphisms in Banyaga's book~\citep{Banyaga1997} for a proof of the following useful fact.

\begin{theorem}\label{thm:moser-volume-form}
    Let $X$ be a compact connected smooth manifold, and let $\omega_0$ and $\omega_1$ be smooth positive volume forms on $X$ with the same total volume.  Then there is a diffeomorphism $\psi:X\to X$ such that $\psi_*\omega_0=\omega_1$. \qed
\end{theorem}

With this volume form theorem in hand, we can transfer uniform expansion.

\begin{figure}[!t]
\centering
\begin{tikzpicture}[scale=0.95,>=Latex,every node/.style={font=\small}]
    \begin{scope}[xshift=0cm]
        \draw[thick,rounded corners=2pt] (0,0) rectangle (3,3);
        \draw[->,thick] (0.35,3.25) -- (1.25,3.25);
        \draw[->,thick] (1.75,-0.25) -- (2.65,-0.25);
        \draw[->,thick] (-0.25,2.65) -- (-0.25,1.75);
        \draw[->,thick] (3.25,1.25) -- (3.25,0.35);
        \draw[gray!55,dashed] (1.5,0.35) -- (1.5,2.65);
        \draw[gray!55,dashed] (0.35,1.5) -- (2.65,1.5);
        \draw[black,line width=0.55pt] (1.42,1.5) -- (1.58,1.5);
        \draw[black,line width=0.55pt] (1.5,1.42) -- (1.5,1.58);
        \foreach \xx/\yy in {0/0,3/0,0/3,3/3}
            \filldraw[fill=blue!70!black,draw=white,line width=0.4pt] (\xx,\yy) circle (3pt);
    \end{scope}

    \draw[->,very thick] (3.85,1.5) -- (5.05,1.5);
    \node at (4.45,1.9) {$x\sim -x$};

    \begin{scope}[xshift=5.85cm]
        \path[draw=black,fill=gray!12,thick]
            (0,0.45)
            .. controls (0,0.05) and (0.35,-0.2) .. (0.8,-0.2)
            -- (2.2,-0.2)
            .. controls (2.65,-0.2) and (3.0,0.05) .. (3.0,0.45)
            -- (3.0,2.55)
            .. controls (3.0,2.95) and (2.65,3.2) .. (2.2,3.2)
            -- (0.8,3.2)
            .. controls (0.35,3.2) and (0,2.95) .. (0,2.55)
            -- cycle;
        \draw[thick] (0.45,1.52) .. controls (1.05,1.25) and (1.95,1.25) .. (2.55,1.52);
        \draw[thick] (0.45,1.68) .. controls (1.05,1.95) and (1.95,1.95) .. (2.55,1.68);
        \foreach \xx/\yy in {0.35/0.18,2.65/0.18,0.35/2.82,2.65/2.82}
            \fill[blue!70!black] (\xx,\yy) circle (2.4pt);
        \node at (1.5,-0.75) {$Q$};
    \end{scope}
\end{tikzpicture}
\caption{A schematic picture of the pillowcase quotient.}
\label{fig:pillowcase}
\end{figure}
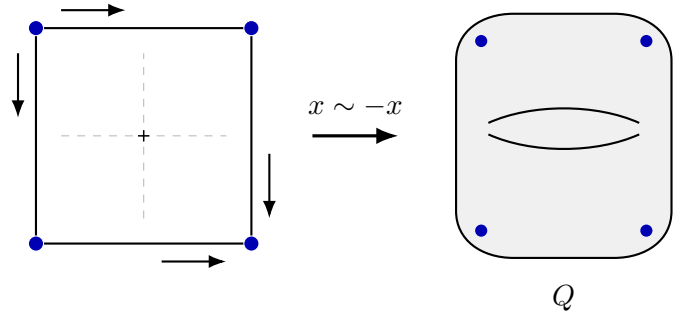

\begin{proposition}\label{prop:S2-gap}
    The two-sphere $S^2$ expands uniformly.
\end{proposition}

\begin{proof}
    Lemma~\ref{lem:pillowcase-expands} gives uniform expansion for the pillowcase.  We now identify the pillowcase with the round sphere up to a measure-preserving bi-Lipschitz change of variables.

    Away from the four cone points, the pillowcase is already a smooth flat surface.  Near a cone point, use polar coordinates on the cone of angle $\pi$.  The angle-doubling map
    \[
        (r,\theta)\mapsto (r,2\theta)
    \]
    straightens the cone into an ordinary Euclidean disc, with only bounded bi-Lipschitz distortion.  It also sends cone area to a constant multiple of Euclidean area.

    After doing this at the four cone points, we may view $Q$ as a smooth sphere equipped with a smooth positive area form and with a smooth Riemannian metric bi-Lipschitz to the pillowcase metric.  By Theorem~\ref{thm:moser-volume-form}, there is a diffeomorphism from this smooth sphere to the round sphere that sends the area form to round area.  Since the sphere is compact, this diffeomorphism is bi-Lipschitz after comparing the two smooth metrics.

    It follows that there is a measure-preserving bi-Lipschitz bijection from $Q$ to the round sphere $S^2$.  The transfer principle, Lemma~\ref{lem:bilip-expansion-transfer}, gives uniform expansion for $S^2$.
\end{proof}

\textbf{Suspending expansion.}
We now prove the higher-dimensional case by induction. Our goal is to prove the following.

\begin{proposition}\label{prop:sphere-suspension}
    Let $m\ge3$.  If $S^{m-1}$ expands uniformly, then so does $S^m$.
\end{proposition}

Of course, Theorem~\ref{thm:main-spheres} follows immediately from Proposition~\ref{prop:sphere-suspension}.

\begin{proof}[Proof of Theorem~\ref{thm:main-spheres} assuming Proposition~\ref{prop:sphere-suspension}]
    $\T^1$ handles $S^1$.  Proposition~\ref{prop:S2-gap} handles $S^2$.  Proposition~\ref{prop:sphere-suspension} then gives $S^m$ for every $m\ge3$ by induction.  The spaces are regular by Lemma~\ref{lem:standard-regular}, so Theorem~\ref{thm:informal-expansion} gives thresholds for all spheres.
\end{proof}

The proof of Proposition~\ref{prop:sphere-suspension} requires one further geometric input, namely a transversality statement on the sphere for two perpendicular height coordinates; we outline this below.

For a fixed $m\ge2$, let $Z=(z_1,\dots,z_{m+1})$ be uniform on $S^m\subset\R^{m+1}$, put
\[
    h_1(z)=z_1,
    \quad
    h_2(z)=z_2,
    \quad
    X=h_1(Z),
    \quad
    Y=h_2(Z),
\]
and let $\nu_m$ be the common law of $X$ and $Y$.  Define $c_m$ to be the Hirschfeld--Gebelein--R\'enyi~\citep{Hirschfeld1935,Gebelein1941,Renyi1959} maximal correlation of the two one-dimensional height variables $X$ and $Y$ defined by
\[
    c_m
    =
    \sup
    \left\|\mathbb E[F(X)\mid Y]\right\|_2,
\]
where the supremum is taken over all $F \in L^2_0([-1,1],\nu_m)$ with $\|F\|_2=1$.

All we need is the fact that $c_m < 1$ for each $m \ge 2$, and a short calculation using the standard change of basis shows this.  A function of a single height coordinate is a zonal function on the sphere.  The spherical-harmonic decomposition is a change of basis for such functions: instead of using an arbitrary function of the height, decompose it into its zonal harmonic pieces, one for each harmonic degree, or equivalently into Gegenbauer polynomials for the one-dimensional height measure. Rotational invariance then implies that conditioning on a perpendicular height coordinate does not mix harmonic degrees; on each degree it acts by a scalar. We are therefore left with a diagonal calculation; standard formulae for the Gegenbauer polynomials imply that $c_m\le 1/m<1$. In fact, it turns out that $c_m = 1/m$ for each $m \ge 2$; see the chapter on the symmetry properties of the Fourier transform in Stein and Weiss~\citep{SteinWeiss1971} for details.

For $n$ points $x=(x_1,\dots,x_n)\in(S^m)^n$ and $i\in\{1,2\}$, let
\[
    \mathcal H_i^{(n)}
    =
    \sigma\left(h_i(x_1),\dots,h_i(x_n)\right),
\]
and let
\[
    E_i^{(n)}f
    =
    \mathbb E\left[f\mid \mathcal H_i^{(n)}\right]
\]
be the corresponding conditional expectation operator on $L^2((S^m)^n)$. In what follows, we suppress the dependence on $n$ because a classical (tensorisation) result of Witsenhausen~\citep{Witsenhausen1975} asserts that passing from one independent pair $(X,Y)$ to the $n$ independent pairs $(h_1(x_j),h_2(x_j))$ does not increase the maximal correlation. The last piece we need to prove Proposition~\ref{prop:sphere-suspension} is the following transversality statement.

\begin{proposition}\label{prop:height-transversality}
    For every $m \ge 2$, $n\ge1$ and $f\in L^2_0((S^m)^n)$, we have
    \[
        \|E_2E_1f\|_2\le (1/m)\|f\|_2.
    \]
\end{proposition}

\begin{proof}
    Let $g=E_1f$.  Then $g$ is mean-zero, $\mathcal H_1^{(n)}$-measurable, and $\|g\|_2\le\|f\|_2$.  Identifying $g$ with a mean-zero function of the vector $(h_1(x_1),\dots,h_1(x_n))$, Witsenhausen's tensorisation theorem gives
    \[
        \|E_2g\|_2\le (1/m)\|g\|_2.
    \]
    Since $E_2g=E_2E_1f$, the claim follows.
\end{proof}

We are now ready to prove that induction by suspension works.

\begin{proof}[Proof of Proposition~\ref{prop:sphere-suspension}]
    The argument proceeds as follows. If $S^{m-1}$ expands, we may apply one of its expanding maps separately on every latitude and leave the height fixed.  This mixes all directions except the height direction.  To mix the height information as well, we do the same thing after rotating the sphere, so that a different coordinate plays the role of height; see Figure~\ref{fig:two-heights} for an illustration. We now fill in the details.

    \begin{figure}[!t]
    \centering
    \begin{tikzpicture}[scale=1.0,>=Latex,every node/.style={font=\small}]
        \begin{scope}[xshift=0cm]
            \draw[thick] (0,0) circle (1.65);
            \draw[dashed] (-1.15,0.85) .. controls (0,1.12) .. (1.15,0.85);
            \draw[blue!70!black,thick] (-1.28,0) -- (1.28,0);
            \draw[dashed] (-1.15,-0.85) .. controls (0,-1.12) .. (1.15,-0.85);
            \draw[->] (0,-2.05) -- (0,2.2) node[above] {$h_1$};
            \node[align=center] at (0,-2.55) {mix on circles of\\constant $h_1$};
        \end{scope}

        \draw[->,very thick] (2.45,0) -- (3.65,0);
        \node at (3.05,0.35) {rotate};

        \begin{scope}[xshift=6.1cm]
            \draw[thick] (0,0) circle (1.65);
            \draw[dashed] (0.85,-1.15) .. controls (1.12,0) .. (0.85,1.15);
            \draw[blue!70!black,thick] (0,-1.28) -- (0,1.28);
            \draw[dashed] (-0.85,-1.15) .. controls (-1.12,0) .. (-0.85,1.15);
            \draw[->] (-2.05,0) -- (2.2,0) node[right] {$h_2$};
            \node[align=center] at (0,-2.55) {then mix on circles of\\constant $h_2$};
        \end{scope}
    \end{tikzpicture}
    \caption{The suspension step, drawn on $S^2$.}
    \label{fig:two-heights}
    \end{figure}
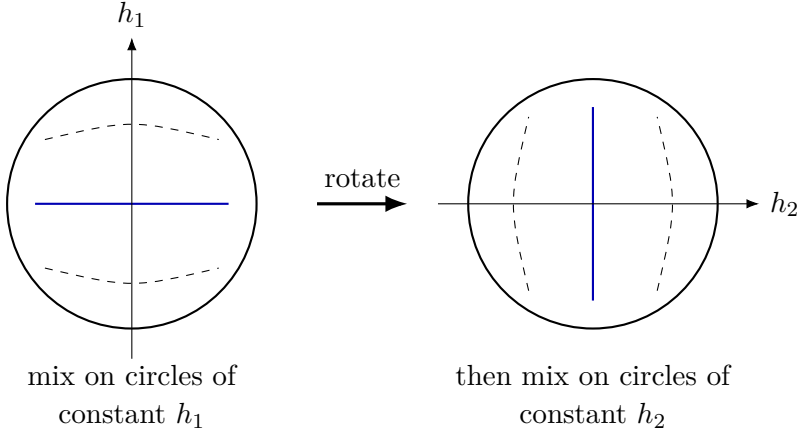

    Write a point of $S^m$ as $(\sin\theta,\cos\theta\,u)$, where $u\in S^{m-1}$ and $\theta \in [-\pi/2,\pi/2]$.  The first-coordinate height is $h_1=\sin\theta$; using $\theta$ instead of $h_1$ is only a reparametrisation of the same height fibres.

    Let the expansion of $S^{m-1}$ be witnessed by maps drawn from a finitely supported probability measure $\mathsf P$, and let $q<1$ be the corresponding contraction factor.  Suspend each map $\phi\in\supp\mathsf P$ to $S^m$ by
    \[
        \widehat\phi(\sin\theta,\cos\theta\,u)
        =
        (\sin\theta,\cos\theta\,\phi(u)).
    \]
    This map preserves spherical measure: conditional on the height $h_1=\sin\theta$, the point $u$ is uniform on the latitude copy of $S^{m-1}$.  It is also bi-Lipschitz, with constants depending only on the old bi-Lipschitz constants, because the round metric in these coordinates is
    \[
        d\theta^2+\cos^2\theta\,ds^2_{S^{m-1}},
    \]
    where $ds^2_{S^{m-1}}$ denotes the round line element on $S^{m-1}$.  Thus, along any curve, $\widehat\phi$ leaves the height contribution to length unchanged and stretches the latitude contribution by at most the old Lipschitz constant; the same applies to $\widehat\phi^{-1}$.

    Let $P_1$ be the averaging operator obtained from these suspended maps, acting diagonally on $(S^m)^n$, and let $E_1$ be conditional expectation onto the $n$ first-coordinate heights.  For fixed heights in $(-1,1)^n$, the fibre is a product of latitude copies of $S^{m-1}$, and on this fibre $P_1$ is exactly the old averaging operator on $(S^{m-1})^n$; the excluded polar heights have measure zero.  Since $f-E_1f$ has mean zero on almost every such fibre, while $E_1f$ is fixed by $P_1$, we get
    \begin{equation}\label{eq:suspension-fiber-mixing}
        \|P_1^k f-E_1f\|_2\le q^k\|f-E_1f\|_2
    \end{equation}
    for all $k\ge1$, uniformly in $n$.  Define $P_2$ and $E_2$ in the same way using the second-coordinate heights.

    Choose $k$ so that $\delta=q^k<(1-1/m)/4$.  Finally set $Q_i=P_i^k$ for $i \in \{1,2\}$.  We show that $Q_2Q_1$ contracts every mean-zero function $f\in L^2_0((S^m)^n)$ with $\|f\|_2=1$.

    From~\eqref{eq:suspension-fiber-mixing}, we see that
    \[
        Q_1f=E_1f+r_1
    \]
    for some $r_1 \in \ker E_1 \subset L^2_0((S^m)^n)$ with $\|r_1\|_2\le\delta$.  Applying the same estimate to $Q_2$ gives
    \[
        Q_2Q_1f=E_2E_1f+E_2r_1+r_2
    \]
    for some $r_2 \in \ker E_2 \subset L^2_0((S^m)^n)$ with $\|r_2\|_2\le\delta$, since $\|Q_1f-E_2Q_1f\|_2\le\|Q_1f\|_2\le1$.

    Proposition~\ref{prop:height-transversality} and $\|E_2r_1\|_2\le\|r_1\|_2$ now give
    \[
        \|Q_2Q_1f\|_2
        \le
        1/m+2\delta
        <1.
    \]

    The operator $Q_2Q_1$ is the average over finitely many compositions of suspended maps and their rotated copies.  Those maps preserve measure and have uniformly bounded bi-Lipschitz constants.  Therefore they witness uniform expansion of $S^m$.
\end{proof}

\section{Conclusion}\label{sec:conclusion}

The geometry of monotone graph properties strikes us as fundamental, and we close with two questions we find interesting.

The first is the basic classification problem, since uniform expansion is only a sufficient condition for the existence of thresholds.

\begin{question}\label{ques:which-spaces}
Which compact metric probability spaces admit thresholds?
\end{question}

It would be useful to find a geometric condition that is both necessary and sufficient  for thresholds to exist. The disconnected one-dimensional space discussed earlier shows that compactness and non-atomicity are not enough.

Second, our proof is deliberately soft and gives essentially no information about the shape of the transition, so it is natural to ask for more precise results.

\begin{question}\label{ques:sharp-thresholds}
Given a metric probability space that admits thresholds, which (sequences of) monotone graph properties have sharp thresholds?
\end{question}

Friedgut~\citep{Friedgut1999} showed for product measures on the hypercube that the absence of a sharp threshold points to a bounded local witness. It would be useful to have a comparable explanation in the geometric setting. Goel--Rai--Krishnamachari~\citep{GoelRaiKrishnamachari2005} bound the additive width of the threshold window, in the spirit of Friedgut--Kalai~\citep{FriedgutKalai1996}, for symmetric monotone properties on cubes and tori. The present paper gives a complementary first step in the multiplicative direction; however, quantitative multiplicative widths in general and sharp thresholds in particular remain largely uncharted.

\section*{Acknowledgements}
The author thanks Will Perkins for introducing the author to the question, Steve Miller for help simplifying the group theory, Jeff Kahn for many interesting discussions about thresholds, and ChatGPT for help with TikZ. The author also acknowledges support from NSF grant DMS-2237138 and a Sloan Research Fellowship.

\bibliographystyle{amsplain}
\bibliography{geometric_graphs}

\end{document}